\documentclass[11pt,a4paper,final]{article}
\usepackage[body={14.6true cm,  22true cm},
            headheight=1.5 em]{geometry}
\usepackage{amsmath}
\usepackage{amsfonts}
\usepackage{graphicx}
\usepackage{listings}
\usepackage{textcomp}
\usepackage{MnSymbol} 
\usepackage{enumerate}
\usepackage{hyperref}
\usepackage{url}
\usepackage{enumerate}

\newenvironment{proof}{\noindent \textbf{Proof.}
}{\unskip\nobreak\hfill\nobreak$\square$\vskip 0.5\baselineskip}

\newtheorem{theorem}{Theorem}[section]
\newtheorem{lemma}{Lemma}[section]

\newtheorem{definition}{Definition}[section]
\newtheorem{remark}{Remark}

\newtheorem{corollary}{Corollary}[section]

 \def\square{\hbox{\vrule height8pt depth0pt
\vbox{\hrule width7.2pt\vskip7.2pt\hrule width7.2pt}\vrule height8pt
depth0pt}\smallskip}

\date{}
\title{ On the index of unbalanced signed bicyclic graphs
{\thanks{Research supported by the Natural Science Foundation of Shanghai (Grant No. 12ZR1420300),   National Natural Science
Foundation of China (No.11101284,11201303 and 11301340).}
}}

\author{  Changxiang He$^1$, Yuying Li$^1$,
Haiying Shan$^2$, Wenyan Wang$^1$\!\normalfont\textsuperscript{\dagger}     \\[5pt]
\small{{  1.  College of Science, University of Shanghai for Science
and Technology, Shanghai 200093,
China}}\\
\small{{ 2.   Department of Mathematics, Tongji University, Shanghai 200092,
China}}  }
\date{}
\begin{document}

\maketitle {\small\bf Abstract:} \,\,\,{\small In this paper,  we focus on the index ( largest eigenvalue) of the adjacency matrix of connected signed graphs. We give some general	results on the index  when the corresponding signed graph is perturbed. As applications,    we determine the first five largest index among all    unbalanced bicyclic graphs on  $n\geq 36$ vertices  together with the corresponding extremal signed graphs whose index attain these values.}
\vspace{2mm}


\section{Introduction}

 Given a simple graph $G=(V(G),\ E(G))$, let $\sigma: E(G) \to	 \{+1,\ -1  \}$ be a mapping defined on the  set  $E(G)$, then we call  $\Gamma= (G, \sigma)$ the {\it signed graph} with {\it underlying graph}  $G$
 and   {\it sign function}  (or {\it signature }) $\sigma$.  
 Obviously,   $G$ and $\Gamma$ share the same set of vertices (i.e. $V(\Gamma) = V(G))$ and have equal
 number of edges (i.e. $\mid E(\Gamma)\mid  =\mid E(G)\mid $).  An edge
 $e$ is {\it positive } ({\it negative}) if $\sigma (e) = +1$ (resp. $\sigma (e) = -1$).
 
 Actually, each concept defined for the underlying graph can be transferred with signed graphs. For example, the degree
 of a vertex $v$ in $G$ is also its degree in $\Gamma$. Furthermore, if some subgraph of the underlying graph is observed, then the sign
 function for the signed subgraph is the restriction of the previous one. Thus, if $v\in  V(G)$, then $\Gamma - v$ denotes the signed
 subgraph having $G-v$ as the underlying graph, while its signature is the restriction from $E(G)$ to $ E(G - v)$ (note, all edges
 incident to $v$ are deleted). Let $U \subset V(G),$ then $\Gamma[U]$ or $G(U)$ denotes the (signed) induced subgraph arising from $U$, while
$ \Gamma - U =\Gamma [V(G)\backslash U
]$.   Let $C$ be a cycle in $\Gamma$, the sign of $C$ is given by
$\sigma(C) =\Pi_{e\in C} \sigma(e)$. A cycle whose sign is $+$ (resp.$-$) is called {\it positive} (resp. {\it negative }). Alternatively,
we can say that a cycle is positive if it contains an even number of negative edges. A signed graph is {\it balanced} if all cycles are
positive; otherwise it is {\it unbalanced}.  There has been a variety of applications of balance, see \cite{Roberts}.


 The {\it adjacency matrix} of a signed graph $\Gamma= (G, \sigma)$ whose vertices are $v_1, v_2,\ldots, v_n$ is the $n\times n$ matrix $A(\Gamma )=(a_{ij})$,
 where
 \begin{equation}
 a_{ij}=\left\{ 
 	\begin{array}{lr}
 \sigma(v_iv_j), & \mbox {if }\  v_iv_j\in E(\Gamma ),\\
  0, &  \mbox{otherwise}.
 \end{array}  
 \right.
 \end{equation}
 
 Clearly, $A(\Gamma )$ is real symmetric and so all its eigenvalues are real. The   characteristic polynomial  $\det( xI -A(\Gamma )) $ of the adjacency
 matrix $A(\Gamma )$ of a signed graph $\Gamma $ is called the {\it characteristic polynomial} of $\Gamma $  and is denoted by $\phi(\Gamma, x) $ . The eigenvalues of $A(\Gamma )$  are
 called the {\it eigenvalues } of $\Gamma $. The largest eigenvalue  is often called the {\it index }, denoted by $\lambda(\Gamma)$.

 Suppose $\theta : V(G)\to \{+1, -1\}$ is any sign function.  {\it Switching } by  $\theta$ means forming a new signed graph  $\Gamma ^{\theta}=(G, \sigma^{\theta})$ 
 whose underlying graph is the same as $G$, but whose sign function is defined on
an edge $uv$ by $\sigma^{\theta} (uv) = \theta(u)\sigma (uv)\theta(v)$. 
Note that switching does not change the signs or balance of the cycles of
 $\Gamma$. If we define a (diagonal) signature matrix $D^{\theta}$ with $d_v=\theta(v)$ for each $v\in V(G),$ then
$A(\Gamma^{\theta}) = D^{\theta}A(\Gamma)D^{\theta}$. Two graphs $\Gamma_1$ and 
$\Gamma_2$ are called {\it switching equivalent}, denoted by
 $\Gamma_1\sim\Gamma_2$, 
 if there exists a switching function $\theta$ such that $\Gamma_2=\Gamma_1^{\theta} $, or equivalently $A(\Gamma_2)=D^{\theta}A(\Gamma_1)D^{\theta}$.
 
 \begin{theorem}\label{thm:swi}\cite{Hou}
 	Let $\Gamma$ be a signed graph. Then $\Gamma $ is balanced if and only if $\Gamma=(G,\sigma)\sim (G,+1)$. 
 \end{theorem}
 Switching equivalence is a relation of equivalence, and two switching equivalent graphs
 have the same eigenvalues. In fact, the signature on bridges is not relevant, hence the edges which do
 not lie on some cycles are not relevant for the signature and they will be always considered as positive.
 
One classical problem of graph spectra is to identify the extremal graphs with respect to the index in some given class of graphs. For signed graphs, since all signatures of a given tree are equivalent, the first non-trivial signature arises for unicyclic graphs, which was considered  in \cite{Akbari}.  The authors determined signed graphs achieving the minimal or the maximal index in the class of unbalanced unicyclic graphs of order $n\geq 3$. In \cite{Fan}, the authors characterized the unicyclic signed graphs of order $n$ with nullity
$n-2$, $n- 3$, $n- 4$, $n- 5$ respectively. For the energy of singed graphs, see \cite{1}, \cite{2}, \cite{Hafeeza},\cite{Pirzada},\cite{Wang},\cite{Wang2} for details.

Here, we will consider unbalanced bicyclic graphs, and  determine the first five largest index among all    unbalanced bicyclic graphs with given order $n\geq 36$  together with the corresponding extremal signed graphs whose index attain these values.

Here is the remainder of the paper. In Section 2, we study the effect of some edges  moving on the index of a signed graph. In Section 3, we introduce the three classes of signed bicyclic graphs. In Section 4, we determine the first five  graphs in the set of    unbalanced bicyclic graphs on  $n\geq 36$ vertices, and order them according to their index in decreasing order.

 \section{Preliminaries}
 The purpose of this section is to analyze how the  index change when modifications are made to a signed graph.  We start with one important tool which also works in  signed graphs. Its general form holds for any principal submatrix of a real symmetric matrix.

\begin{lemma}\label{th:interlacing}
	 (Interlacing theorem for signed graphs). Let   $\Gamma= (G, \sigma)$ be a signed graph of order $n$ and $\Gamma -v$ be the signed graph obtained from  $\Gamma$ by deleting the vertex $v$. If $\lambda_i$ are the (adjacency) eigenvalues, then
 $$\lambda_1(\Gamma)\geq \lambda_1(\Gamma-v)\geq \lambda_2(\Gamma)\geq \lambda_2(\Gamma-v)\geq \ldots  \geq \lambda_{n-1}(\Gamma-v)\geq \lambda_n(\Gamma). $$ 
\end{lemma}

\begin{lemma}\label{lem:cut}
	Let $\Gamma$   be a signed graph with cut edge $uv$,    and  ${\bf x}$ be an   eigenvector corresponding to the index $\lambda(\Gamma)$.  We have $\sigma(uv)x_{u}x_v\geq 0$.  
\end{lemma}
\begin{proof}
Without loss of generality, we assume that ${\bf x}$ is unit and $\sigma(uv)>0$.   By way of contradiction, we suppose that $ x_{u}x_v<0$.	Let  $\Gamma _1$ and $\Gamma_2 $ be the two connected components  of $\Gamma -uv$, respectively.  Set ${\bf x}=\left( \begin{matrix}
	
	{\bf x}_1\\
	
	{\bf x}_2
	
\end{matrix} \right)$, where $	{\bf x}_1$ and  $	{\bf x}_2$ are  the subvectors of ${\bf x}$   indexed by vertices in  $\Gamma_1$ and $\Gamma_2$, respectively. Let  ${\bf y}=\left( \begin{matrix}

-{\bf x}_1\\

{\bf x}_2

\end{matrix} \right)$, then ${\bf y}^{T}A(\Gamma){\bf y}-{\bf x}^{T}A(\Gamma){\bf x}=-4 x_{u}x_v> 0$, which contradicts to the fact that   ${\bf x}$ maximizes the Rayleigh quotient.  
\end{proof}

From the above lemma, it is straightforward to derive the   following result.
\begin{corollary} 
Let $T$ be a    vertex induced subtree in the signed graph $\Gamma $,  and  ${\bf x}$ be an   eigenvector corresponding to the index $\lambda(\Gamma)$. Then for any edge  $uv$  of  $T$, we have $\sigma(uv)x_{u}x_v\geq 0$. 
\end{corollary}

\begin{remark}\label{rem:tree}
If $T$ is  a    vertex induced subtree  with root $v$ in signed graph $\Gamma $, the above corollary implies that if $x_v\geq 0$  we can assume that all edges in $T$ are positive and all vertices of  $T$ have  non-negative coordinates in ${\bf x}$. This is valid because we can prove it by  using switching  equivalent from the leaves of the rooted subtree.
\end{remark}

We proceed by considering how the index change when cut edges be moved.
\begin{lemma}\label{lem:gra1}
	Let $u$, $v$ be two vertices of the signed graph $\Gamma$,    $vv_1,\ldots,vv_s\ (s\geq 1)$ be cut edges of    $\Gamma$, and ${\bf x}$ be an eigenvector corresponding to $\lambda(\Gamma)$. Let $$\Gamma'=\Gamma-vv_1-\ldots-vv_s+uv_1+\ldots+uv_s.$$ If $x_u \geq  x_v \geq 0$ or  $x_u \leq  x_v \leq 0$, we have   $\lambda(\Gamma')\geq  \lambda(\Gamma)$. 
\end{lemma}
\begin{proof}
 Without loss of generality, we assume that ${\bf x}$ is unit. Due to the Rayleigh quotient, we   have    
$$\lambda(\Gamma')-\lambda(\Gamma)\geq  {\bf x}^T A(\Gamma') {\bf x}-{\bf x}^T A(\Gamma) {\bf x}=(x_u-x_v) \sum\limits_{i=1}^s \sigma(vv_i)x_{v_i}.$$

Lemma \ref{lem:cut} tell us that $\sigma(vv_i)x_{v_i}x_{v}\geq 0$, 
one can quickly verify that $\lambda(\Gamma')\geq \lambda(\Gamma)$ when $x_u \geq  x_v \geq 0$ or  $x_u \leq  x_v \leq 0$.
\end{proof}

If $vv_1,\ldots,vv_s $ are pendant edges  in the above lemma, the eigenvalue equation leads to  $\lambda(\Gamma) x_{v_i} =\sigma(vv_i) x_v$, which implies that $\sigma(vv_i) x_v x_{v_i}>0$  when $x_v\neq 0$, so we can get a stronger version of the above result.

\begin{lemma}\label{lem:gra2}
	Let $u$, $v$ be two vertices of signed graph $\Gamma$,     $vv_1,\ldots,vv_s\ (s\geq 1)$ be pendant edges of    $\Gamma$, and ${\bf x}$ be an eigenvector corresponding to $\lambda(\Gamma)$. Let $$\Gamma'=\Gamma-vv_1-\ldots-vv_s+uv_1+\ldots+uv_s.$$ If $x_u \geq  x_v \geq 0$ or  $x_u \leq  x_v \leq 0$, we have   $\lambda(\Gamma')\geq  \lambda(\Gamma)$. Furthermore, if  $x_u >  x_v > 0$ or  $x_u < x_v < 0$, then   $\lambda(\Gamma')>  \lambda(\Gamma)$.
\end{lemma}

In Lemma \ref{lem:gra1} and Lemma \ref{lem:gra2}, the edges be moved are all cut edges. Now the perturbation, $\alpha $-transform, described in the following  can be seen in many books  and many other papers, which can move non-cut edges from one vertex to another.
 
  \begin{definition} Let  $\Gamma$ be a connected signed  graph, $uv$ be a non-pendant
 	edge of $\Gamma$ which is not in any triangle. Let $N_{\Gamma}(v)\backslash
 	\{u\}=\{v_1,\cdots, v_d\}$ with $   d\geq1$. The signed  graph
 	$$\Gamma'=\alpha(\Gamma,uv)=\Gamma-vv_1-vv_2-\cdots-vv_d+uv_1+uv_2+\cdots+uv_d.$$ We say that $\Gamma'$ is an
 	$\alpha$-transform of $\Gamma$ on the edge $uv$. 
 	\end{definition}
 All  edges retain the sign
 they have after $\alpha$-transform. In the next, we focus on how the index changes  after  $\alpha$-transform.
 
\begin{lemma}\label{lem:con}
	Let $uv$ be an edge of signed graph $\Gamma$, and $\Gamma'=\alpha(\Gamma, uv)$ be the graph obtained from  $\Gamma$ by $\alpha$-transform on the edge $uv$. Let  $\bf{x}$ be an   eigenvector corresponding to $\lambda(\Gamma)$. If one of the following condition holds, we have $\lambda(\Gamma')\geq \lambda(\Gamma)$:
	\begin{enumerate}[(1).]
		\item   if $\sigma(uv)>0$, and $x_v\leq x_u\leq \lambda(\Gamma )x_v$,
		\item  if $\sigma(uv)<0$ and $x_u\geq 0,\ x_v\geq 0$.
	\end{enumerate}
Furthermore, if one of the following can be satisfied:
 \begin{enumerate}[(1).]
	\item   if $\sigma(uv)>0$, and $x_v< x_u< \lambda(\Gamma )x_v$,
	\item  if $\sigma(uv)<0$, and $x_u>0,\ x_v> 0$ and $x_u\neq x_v$,
\end{enumerate}
we have  $\lambda(\Gamma')> \lambda(\Gamma)$.
\end{lemma}
\begin{proof}
  Let $N_{\Gamma}(u)\backslash \{v\}=\{u_1,\ldots,u_r\}$ 
 and  $N_{\Gamma}(v)\backslash \{u\}=\{v_1,\ldots,v_s\}$. 
 The eigenvalue equation leads to the relation 
 $$\lambda(\Gamma)x_v=\sigma(uv)x_u+\sum\limits_{v_i\in N_{\Gamma }(v)\backslash \{u\}} \sigma(vv_i)x_{v_i},$$
  $$\lambda(\Gamma)x_u=\sigma(uv)x_v+\sum\limits_{u_i\in N_{\Gamma }(u)\backslash \{v\}} \sigma(uu_i)x_{u_i}.$$
  
  These      then easily imply that
  \begin{align} \lambda(\Gamma')-\lambda(\Gamma)&\geq {\bf x}^T A(\Gamma') {\bf x}-{\bf x}^T A(\Gamma) {\bf x}=(x_u-x_v)\sum\limits_{v_i\in N_G(u)\backslash \{v\}} \sigma(vv_i)x_{v_i}\\
  &=(x_u-x_v)(\lambda(\Gamma)x_v-\sigma(uv)x_u),
  \end{align} 
  and \begin{align}\lambda(\Gamma')-\lambda(\Gamma)&\geq {\bf x}^T A(\Gamma') {\bf x}-{\bf x}^T A(\Gamma) {\bf x}=(x_v-x_u)\sum\limits_{u_i\in N_G(v)\backslash\{u\}} \sigma(uu_i)x_{u_i}\\
&= (x_v-x_u)(\lambda(\Gamma)x_u-\sigma(uv)x_v). 
\end{align}

So that if  $\sigma(uv)>0$, applying (3), we estimate that $$\lambda(\Gamma')-\lambda(\Gamma)\geq (x_u-x_v)(\lambda(\Gamma)x_v-x_u). $$ Thus, $\lambda(\Gamma')\geq\lambda(\Gamma)$ when  $x_v\leq x_u\leq \lambda(\Gamma )x_v$, the inequality is strict when  $x_v< x_u< \lambda(\Gamma )x_v$.

If $\sigma(uv)<0$, it seems more complicated. By (5), we know $$\lambda(\Gamma')-\lambda(\Gamma)\geq (x_u-x_v)(\lambda(\Gamma)x_v+x_u).$$ 
The symmetry tell us that we also have $$\lambda(\Gamma')-\lambda(\Gamma)\geq (x_v-x_u)(\lambda(\Gamma)x_u+x_v).$$  Therefore, if $x_u\geq 0,\ x_v\geq 0$, then   $\lambda(\Gamma')\geq\lambda(\Gamma)$  whenever $x_u\geq v_v$ or $x_u<x_v$, and  the inequality is strict when   $x_u>0,\ x_v> 0$ and $x_u\neq x_v$.
\end{proof}

In all figures, solid   and dotted edges represent positive and negative edges, respectively.

\begin{figure}[!htb] 
	\begin{center} 
		\includegraphics{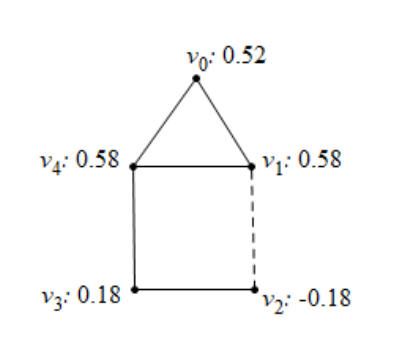}
		\makeatletter\def\@captype{figure}\makeatother 
		\vspace{-0.5cm}
		\caption{ The  example $\Gamma$ in Remark 2}\label{cont}
	\end{center}
\end{figure}

\begin{remark}
	The conditions in Lemma \ref{lem:con} are necessary. For example, the    signed graph $\Gamma $ (as shown in Figure \ref{cont}) with index $\lambda(\Gamma) \approx 2.214$,  its positive edge $v_2v_3$ does not satisfy the condition in  Lemma \ref{lem:con}. If we let $\Gamma'=\alpha(\Gamma, v_2v_3)$, then the index $\lambda(\Gamma')=2$   is less than $\lambda(\Gamma')$.
\end{remark}

However, in Lemma \ref{lem:con}, if $uv$ is a cut edge, things are easier.
\begin{corollary}\label{lem:concut}
	Let $uv$ be a   cut  edge of signed graph $\Gamma$,  and $\Gamma'=\alpha(\Gamma, uv)$. We have  $\lambda(\Gamma')\geq \lambda(\Gamma)$.
\end{corollary}
\begin{proof}
Suppose, without loss of generality,    that  $uv$ is positive. Let ${\bf x}$ be an unit eigenvector corresponding to $\lambda(\Gamma)$. 	By Lemma \ref{lem:cut}, we can assume that $x_u\geq x_v\geq 0$. 
	Let  $N_{\Gamma}(v)\setminus \{u\}=\{v_1,\ldots,v_s\}$. 
	The eigenvalue equation leads to the relation 
	$$\lambda(\Gamma)x_v= x_u+\sum\limits_{v_i\in N_{\Gamma }(v)\setminus \{u\}} \sigma(vv_i)x_{v_i}.$$
	 We claim that $\sum\limits_{v_i\in N_{\Gamma }(v)\setminus \{u\}} \sigma(vv_i)x_{v_i}\geq 0$. Otherwise, we write     the component of $\Gamma -uv$      containing  the vertex $v$ as $U$.   Set ${\bf x}=\left( \begin{matrix}
	
	{\bf x}_1\\
	
	{\bf x}_2
	
\end{matrix} \right)$, where $	{\bf x}_1$ is  the subvectors of ${\bf x}$   indexed by vertices in  $U-v$.   Let ${\bf y}=\left( \begin{matrix}

-{\bf x}_1\\

{\bf x}_2

\end{matrix} \right)$, then ${\bf y}^{T}A(\Gamma){\bf y}-{\bf x}^{T}A(\Gamma){\bf x}=-4 x_v\sum\limits_{v_i\in N_{\Gamma }(v)\setminus \{u\}} \sigma(vv_i)x_{v_i} > 0$, which contradicts to the fact that   ${\bf x}$ maximizes the Rayleigh quotient. 

Since  $\sum\limits_{v_i\in N_{\Gamma }(v)\setminus \{u\}} \sigma(vv_i)x_{v_i}\geq 0$,  we have $x_u\leq \lambda (\Gamma ) x_v$. By gluing together this inequality with $x_u\geq x_v$ and Lemma \ref{lem:con}, we  get the assertion. 
	\end{proof}

  The above lemma tell us that if $T$ is a vertex induced subtree of signed graph with root $v$, then  $\alpha $-transform on any edge in $T$ will not decrease the index of the  signed graph. Thus, replacing $T$ with a star ( with   center $v$ and order $\mid V(T)\mid $ ) will not decrease the index as well.

We recall from \cite{Belardo1} the following Schwenk's  formulas
\begin{lemma} \label{lem:Schwenk} Let $v$ be a vertex of signed graph $\Gamma$,  
	$$\Phi(\Gamma,x)=x\Phi(\Gamma-v,x)-\sum\limits_{uv\in E(\Gamma)}\Phi(\Gamma-u-v,x)-2\sum\limits_{C\in \mathcal{C}_v}\sigma(C)\Phi(\Gamma-C,x),$$
	where $ \mathcal{C}_v$ is  the set of signed cycles passing through $v$, and $\Gamma-C$ is the signed	graph obtained from
	$\Gamma$ by deleting   $C$.
\end{lemma}

\section{Three classes of signed bicyclic graphs}

A graph $G$ of order $n$ is called a {\it bicyclic graph } if $G$ is
	connected and the number of edges of $G$ is $n+1$. A signed graph whose  underlying graph is a bicyclic graph, we  call it  {\it signed  bicyclic  graph}.

 It is easy to see from the definition that $G$ is a   bicyclic graph 
 if and only if $G$ can be obtained from a tree $T$ (with the same
 order) by adding two new edges to $T$.

 Let $G$ be a bicyclic graph. The {\it base of bicyclic graph}    $G$, denoted by
 $\widehat{G}$, is the (unique) minimal bicyclic subgraph of $G$. If $\Gamma=(G,\sigma)$, then we define $\widehat{\Gamma}=(\widehat{G},\sigma)$ as the {\it base of signed  bicyclic graph}   $\Gamma $. It
 is easy to see that $\widehat{G}$ is the unique bicyclic subgraph of
 $G$  containing no pendant vertices, while $G$ can be obtained from
 $\widehat{G}$ by attaching trees to some vertices of $\widehat{G}$.
 
 It is well-known that there are the following three types of
 bicyclic graphs containing no pendant vertices:

 Let $B(p,q)\ (p\geq q\geq 3)$ be the bicyclic graph obtained from two vertex-disjoint
 cycles $C_p$ and $C_q$ by identifying vertices $u$ of $C_p$ and $v$
 of $C_q$\ (see Fig. 2.1). This type of graph is also known as the {\it infinity graph}.
 
 Let $B(p,\ell ,q)$ be the bicyclic graph obtained from two
 vertex-disjoint cycles $C_p$ and $C_q$ by joining vertices $u$ of
 $C_p$ and $v$ of $C_q$ by a new path $uu_1u_2\cdots u_{\ell -1}v$ with
 length $\ell \ (\ell \ge 1)$\ (see Figure \ref{bpq}). This type of graph is also known as the {\it dumbbell graph}; if the cycles are triangles, it also takes the name of {\it hourglass graph}.
 
 \begin{figure}[!htb] 
 	\begin{center} \hspace{-2cm}
 		\includegraphics[width=12cm]{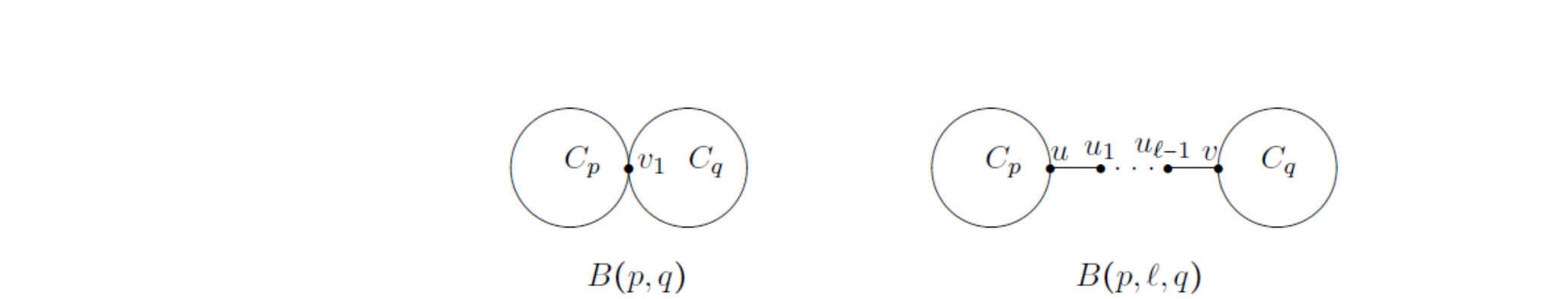}
 		\makeatletter\def\@captype{figure}\makeatother 
 		\caption{ $B(p, q)$ and $B(p, \ell , q)$ }\label{bpq}
 	\end{center}
 \end{figure}

 
 Let $B(P_k,P_{\ell},P_m)\ (1\leq m\leq\min\{k,\ell \})$ be the bicyclic graph
 obtained from three  pairwise internal  disjoint paths form a vertex
 $x$ to a vertex $y$. These three paths are $ xv_1v_2\cdots,\
 v_{k-1}y$ with length $k$, $ xu_1u_2\cdots,\ u_{\ell -1}y$ with length
 $\ell $ and $xw_1w_2\cdots,\ w_{m-1}y$ with length $m$\ (see Figure \ref{bpql}). This type of graph is also known as the {\it $\theta$-graph}.
 \begin{figure}[!htb] 
 	\begin{center} 
 		\includegraphics[width=12cm]{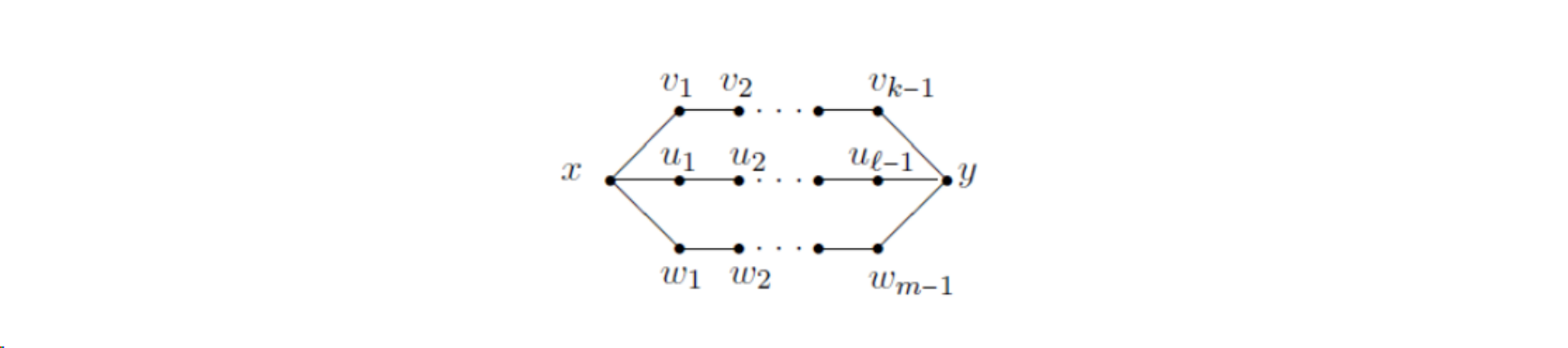}
 		\makeatletter\def\@captype{figure}\makeatother 
 		\vspace{-0.5cm}
 		\caption{ $B(P_k,P_{\ell},P_m)$ }\label{bpql}
 	\end{center}
 \end{figure}

  Accordingly, we denoted
 by $\mathcal{B}_n$ the set of all   unbalanced signed bicyclic graphs of order $n$.
 We are now ready to describe the class of  unbalanced signed bicyclic graphs.

 $\mathcal{B}_n (p,q) =\{\Gamma=(G,\sigma) \mbox{ is \ unbalanced}   \mid \ \widehat{G}=B(p,q)\,\,
 \mbox{ for\ some\ }    p \geq q \geq 3 \}, $
 
 $ \mathcal{B}_n (p,\ell,q) =\{\Gamma=(G,\sigma) \mbox{ is \ unbalanced}  
 \ \mid\  \widehat{G}=B(p,\ell,q), \mbox{ for\ some\ }  p \geq 3,  \, q
 \geq 3 \,\, \mbox{and}  \,\, \ell  \geq 1\}, $

 $\mathcal{B}_n (P_k,P_{\ell},P_m) =\{\Gamma=(G,\sigma) \mbox{ is \ unbalanced}  \ \mid\
 \widehat{G}=B(P_k,P_{\ell},P_m) \mbox{ for\ some\ }  1\leq
 m\leq\min\{k,l\}\}.$
 
 It is easy to see that
 $$\mathcal{B}_n=\mathcal{B}_n (p,q) \, \dot{\cup} \, \mathcal{B}_n (p,\ell,q)\, \dot{\cup}\, \mathcal{B}_n (P_k,P_{\ell},P_m).$$

 \section{The index of unbalanced signed bicyclic graphs with  given order }
 
 In this section,   we deal with the extremal index problems for the class of unbalanced signed bicyclic  graphs with order $n$. We will determine the first five  graphs in  $\mathcal{B}_n$, and  order them  according to their index in decreasing order.
 
 For the unicyclic graphs, there are exactly two
 switching equivalent classes. If a unicyclic signed graph is balanced, by Theorem \ref{thm:swi}, it is switching
 equivalent to one with all edges positive. Otherwise, it is switching equivalent to one with exactly one (arbitrary) negative edge on the cycle\cite{Fan}. For unbalanced bicyclic signed graphs, we also have similar results.
 
 \begin{lemma}\label{lem:neg}
 	If $\Gamma \in \mathcal{B}_n (p,q) \cup \mathcal{B}_n (p,\ell,q)$, then $\Gamma $ is switching equivalent to one with exactly one (arbitrary) negative edge on its  unbalanced cycle. If $\Gamma\in \mathcal{B}_n (P_k,P_{\ell},P_m)$,  then $\Gamma $ is switching equivalent to one with exactly one (arbitrary) negative edge on its base.
 \end{lemma}
 \begin{proof}
 	 If $\Gamma \in \mathcal{B}_n (p,q) \cup \mathcal{B}_n (p,\ell,q)$, 
 	let $e_1$ and $e_2$ be two edges of $\Gamma$ in different cycles, then $\Gamma-e_1-e_2$ is a tree, which is   balanced. So by Theorem \ref{thm:swi},
 	there exists a sign function $\theta$ such that $(\Gamma-e_1-e_2) ^{\theta }$ consisting of positive edges. Returning to the graph $\Gamma ^{\theta}$, the edges $e_1$ and $e_2$ must  have a negative sign as switching does not change the sign of a cycle.

 	If $\Gamma\in \mathcal{B}_n (P_k,P_{\ell},P_m)$, let $e_1$, $e_2$ and $e_3$ be the three edges of $\Gamma$ which are incident to a common 3-degree vertex in the base. Similarly, $(\Gamma-e_1-e_2) ^{\theta }$ consisting of positive edges. Returning to the graph $\Gamma ^{\theta}$, if exactly one of    $e_1$ and $e_2$ is negative,  the result follows.  If both  $e_1$ and $e_2$ are negative, then $\Gamma$ is switching equivalent to the signed graph which has the same underlying graph as $\Gamma$, and  just has one negative edge $e_3$.
 \end{proof}

 The following lemma is a starting point of our discussions.
 \begin{lemma}\label{lem:albi}
Let $u_1u_2u_3u_4$ be a path in   signed bicyclic graph $\Gamma$, and  $ d_{\hat{\Gamma} }(u_2)=d_{\hat{\Gamma} }(u_3)=2$. Let ${\bf x}$ be an eigenvector corresponding to the index  $\lambda(\Gamma)$ and $\Gamma'=\alpha(\Gamma,u_2u_3)$. If  $x_{u_2}\geq 0,\ x_{u_3}\geq 0$,  $ \sigma(u_1u_2)x_{u_1}\geq0$ and $ \sigma(u_3u_4)x_{u_4}\geq0$,   then   $\lambda(\Gamma')\geq \lambda(\Gamma)$.
\end{lemma}
\begin{proof} From Lemma \ref{lem:con}, it suffices to consider the case that $u_2u_3$ is a positive edge.
	
If $x_{u_2}\leq x_{u_3}$,  the eigenvalue equation for the index $\lambda(\Gamma)$, when restricted to the vertex $u_2$ becomes				 $$\lambda(\Gamma)x_{u_2}=\sigma(u_1u_2)x_{u_1}+\sum\limits_{v_i\in N_{\Gamma}(u_2)\setminus \{u_1,u_3\}} \sigma(u_2v_i)x_{v_i}+ x_{u_3}.$$
The fact that $\Gamma$ is a signed bicyclic graph and 
$ d_{\hat{\Gamma} }(u_2) =2$ imply that $u_2v_i$ is a cut edge, and then $\sigma(u_2v_i)x_{v_i}\geq 0$  follows from Lemma \ref{lem:cut}. Hence, $x_{u_3}\leq \lambda(\Gamma)x_{u_2} $. By Lemma \ref{lem:con}, we can get the  desired result.

Similarly, we can prove the assertion when   $x_{u_2}\geq x_{u_3}$.
\end{proof}
	
For convenience, we use $\Gamma+\widetilde{uv}$ (where $uv\not\in E(\Gamma)$) to denote the signed graph obtained from $\Gamma$ by adding a new negative edge $uv$. 
 
\begin{lemma}\label{lem:bibase} Let $\Gamma$ be a $\infty$-type unbalanced signed  bicyclic graph,  and $\hat{\Gamma}\not\in \mathcal{B}_n(3,3) $,  then there is some $\infty$-type unbalanced signed  bicyclic  graph $\Gamma'$ such that $|V(\hat{\Gamma'})|<|V(\hat{\Gamma})|$ and $\lambda(\Gamma')\geq \lambda(\Gamma)$.
\end{lemma}
\begin{proof} 
By Lemma \ref{lem:neg}, we can assume     that there is exactly one negative edge in an unbalanced cycle, and all edges in balanced cycle are positive.
	
	 Let  $u_1u_2\ldots u_{g_1}$ be the   unbalanced  cycle of $\Gamma $ with larger length, $u_1u_2$  be its unique negative edge,     and again ${\bf x}$ be an unit eigenvector corresponding to $\lambda(\Gamma)$.  Without loss of generality, we assume $x_{u_1}\geq 0$. 
	
	 If $g_1=3$. Let $u_1u'_2\ldots u'_{g_2}\ (g_2\geq 4)$ be another cycle of $\Gamma $, note that $u_1u_2u_3$ is the unbalanced cycle with larger length, and $\hat{\Gamma}\not\in \mathcal{B}_n(3,3) $, we find that  $u_1u'_2\ldots u'_{g_2}$  is balanced. We claim that    the subvector ${\bf x}_1$ of ${\bf x}$ indexed by vertices in  the cycle $u_1u'_2\ldots u'_{g_2}$ is nonnegative. Otherwise, let ${\bf y}$ be the vector obtained from ${\bf x}$ by replacing all negative entries in ${\bf x}_1$ with their absolute, then ${\bf y}^{T}A(\Gamma ){\bf y}\geq {\bf x}^{T}A(\Gamma ){\bf x}$, with equality if and only if ${\bf y}$ is also an eigenvector of $\lambda(\Gamma)$. Then we can either get  the claim (by choosing ${\bf x} $ as ${\bf y}$) or a contradiction (contradicts to the fact that  ${\bf x}^{T}A(\Gamma ){\bf x}$ maximizes the Rayleigh quotient). Note that $g_2\geq 4$, we can get the desired $\Gamma'$ by using $\alpha$-transform on the edge $u'_2u'_3$. Therefore, in the next, we assume that $g_1\geq 4$.

	If all non-zero elements in $\{x_{u_3},\ x_{u_4},\ldots, x_{u_{g_1}}\}$  have the same sign, we can get    the desired unbalanced signed graph by Lemma \ref{lem:albi}. Now we consider the case that $\{x_{u_3},\ldots, x_{u_{g_1}}\}$ have different signs.
	
If $x_{u_2}\geq 0,\ x_{u_{3}}\leq 0$, then  $\Gamma'=\Gamma-u_2u_{3}+\widetilde{u_1u_3}$ is the desired unbalanced signed graph with unbalanced cycle $u_1u_3\ldots u_{g_1}$.	If there is some edge $u_iu_{i+1}$, where $3\leq i \leq g_1-1$,  such that $x_{u_i}\geq 0,\ x_{u_{i+1}}\leq 0$, then  $\Gamma'=\Gamma-u_iu_{i+1}+u_1u_{i}$ is the desired unbalanced signed graph with unbalanced cycle $u_1u_2\ldots u_{i}$.   

To complete the proof, it suffices to consider the case that  there is some $3\leq s \leq g_1$ such that $x_{u_2}\leq 0,\ldots, x_{u_s}\leq 0$ and $x_{u_{s+1}}\geq 0,\ldots, x_{u_{g_1}}\geq 0$. If $g_1\geq 5$, as the larger of  $s-1$ and $g-(s-1) $ is at least half of $g_1$ (which is equal to or greater than  3),  so we can get the desired $\Gamma'$ by Lemma \ref{lem:albi}. It remains to consider the case that $g_1=4$ and  $x_{u_2}\leq 0, \  x_{u_3}\leq 0,\  x_{u_4}\geq 0$.  By using the  switching equivalent, we can get a signed graph with all non-negative entries corresponding to $\lambda(\Gamma)$. By  using Lemma \ref{lem:albi} again, we can get the desired result.
\end{proof}

\begin{lemma}\label{lem:tribase} Let $\Gamma$ be a $\theta$-type unbalanced  signed bicyclic   graph,  and $\hat{\Gamma}\not\in \mathcal{B}_n(P_1,P_2,P_2) $,  then there is some $\theta$-type unbalanced  signed  bicyclic   graph $\Gamma'$ such that $|V(\hat{\Gamma'})|<|V(\hat{\Gamma})|$ and $\lambda(\Gamma')\geq \lambda(\Gamma)$.
\end{lemma}
\begin{proof} 
Suppose, without loss of generality,    that there is just one negative edge in the base.

Let $u_1$ be one of the 3-degree vertices of $\hat{\Gamma}$,     $u_1u_2$  be the unique negative edge. Again  let  ${\bf x}$ be an unit eigenvector corresponding to $\lambda(\Gamma)$ with $x_{u_1}\geq 0$. 

If $x_{u_2}\geq 0$, similar to the proof of the case $g_1=3$ in Lemma \ref{lem:bibase},  $\bf{x}$ is nonnegative,  we can get the desired $\Gamma'$ by using $\alpha$-transform.

Consequently, if $x_{u_2}< 0$. Let $u_1u_2'\ldots u_p'u_2$ be the longest path from $u_1$ to $u_2$. If there is some edge $u_i'u_{i+1}'$ such that $x_{u_i'}\leq 0$,  $x_{u_{i+1}'}\geq 0$, then $\Gamma'=\Gamma-u_i'u_{i+1}'+u_2u_i'$ is the desired signed graph. If there is some edge $u_i'u_{i+1}'$ such that $x_{u_i'}\geq 0$,  $x_{u_{i+1}'}\leq 0$, then $\Gamma'=\Gamma-u_i'u'_{i+1}+ \widetilde{u_2u'_i}$ is the desired signed graph. 
If all non-zero entries in $x_{u'_2},\ldots,x_{u'_p}$ have the same sign, as before,   we can set  $\Gamma'=\alpha(\Gamma, u_2'u_3')$.
\end{proof}

\begin{figure}[!htb] 
	\begin{center}
		\includegraphics{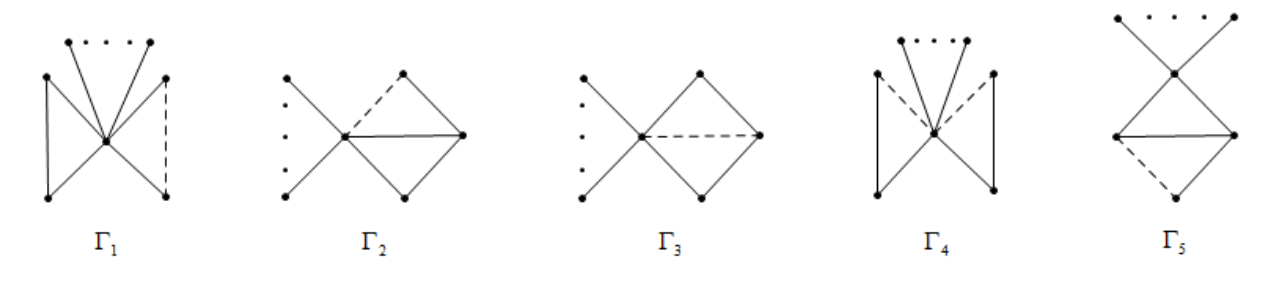}
		\makeatletter\def\@captype{figure}\makeatother 
		\vspace{-0.5cm}
		\caption{ Five   signed  graphs with  maximum index in $\mathcal{ B}_n$ }\label{fig:T1-T5}
	\end{center}
\end{figure}
\begin{lemma}\label{lem:ind} Let $\Gamma_i \in \mathcal{B}_n\ ( \mbox{ where} \  i=1,2,\ldots, 5)$ be the unbalanced signed graphs as shown in Figure  \ref{fig:T1-T5}, 
	then $\lambda(\Gamma_i)$ is the largest root of the equation $f_i(x)=0$, where  
	\begin{align*}
	f_1(x)&=x^4-nx^2+n-5, \\
	f_2(x)&=x^4-(n+1)x^2+2n-4,  \\
	f_3(x)&=x^4-(n+1)x^2+4x+2n-8, \\
	f_4(x)&=x^3+x^2-(n-1)x-n+5, \\
	f_5(x)&=x^3-x^2-(n-2)x+n-4.
	\end{align*}
	Furthermore, we have $\lambda(\Gamma_1)>\lambda(\Gamma_2)>\lambda(\Gamma_3)>\lambda(\Gamma_4)>\lambda(\Gamma_5)$ when $n\geq 36$.
		 \end{lemma}
\begin{proof}
	By Lemma \ref{lem:Schwenk}, one can get the characteristic polynomials of $\Gamma_1, \Gamma_2,\Gamma_3,\Gamma_4,\Gamma_5$ by direct calculation,
\begin{align*}
\Phi(\Gamma_1,x)&=x^{n-6}(x^2-1)(x^4-nx^2+n-5), \\
\Phi(\Gamma_2,x)&=x^{n-4}[x^4-(n+1)x^2+2n-4],  \\
\Phi(\Gamma_3,x)&=x^{n-4}[x^4-(n+1)x^2+4x+2n-8], \\
\Phi(\Gamma_4,x)&=x^{n-6}(x+1)(x-1)^2[x^3+x^2-(n-1)x-n+5], \\
\Phi(\Gamma_5,x)&=x^{n-5}(x+2)(x-1)[x^3-x^2-(n-2)x+n-4].
\end{align*}	
By comparing the index of   graphs  and applying equations above, we have \begin{align*}
\Phi(\Gamma_2,x)-\Phi(\Gamma_1,x)&=x^{n-6}(x^2+n-5)>0,\\
\Phi(\Gamma_3,x)-\Phi(\Gamma_2,x)&=4x^{n-4}(x-1),\\
\Phi(\Gamma_4,x)-\Phi(\Gamma_3,x)&=x^{n-6}(3x^2-4x-n+5).
\end{align*}	
The Interlacing Theorem implies that $\lambda(\Gamma_i)>\sqrt{n-2}>1$ for $i=2,3$. It is not difficult to see that  	$\Phi(\Gamma_3,x)>\Phi(\Gamma_2,x)$ when $x\geq \lambda(\Gamma_2)$ and $\Phi(\Gamma_4,x)>\Phi(\Gamma_3,x)$ when $x\geq \lambda(\Gamma_3)$. These are exactly what we need here,  	$\lambda(\Gamma_1)>\lambda(\Gamma_2)>\lambda(\Gamma_3)>\lambda(\Gamma_4)$.
	
To compare $\lambda(\Gamma_4)$ and $\lambda(\Gamma_5)$, we let
\begin{align*}f_4(x)&=x^3+x^2-(n-1)x-n+5,\\
f_5(x)&=x^3-x^2-(n-2)x+n-4.
\end{align*}\
Then	 $g(x)=f_4(x)-f_5(x)=2x^2-x-2n+9$ has the largest root $\frac{1+\sqrt{16n-71}}{4}$. One can check directly   $f_5(-\infty)<0$,  $f_5(0)=n-4>0$,  $f_5(1)=-2<0$ and $f_5(\frac{1+\sqrt{16n-71}}{4})>0$ when $n\geq 36$. Hence, the largest root of $f_5(x)=0$ is less than $\frac{1+\sqrt{16n-71}}{4}$, which implies that $f_4(x)<0$ when $x$ is the largest root of $f_5(x)=0$. Therefore, we have $\lambda(\Gamma_4)>\lambda(\Gamma_5)$. This completes the proof. 
\end{proof}

\begin{figure}[!htb] 
	\begin{center}	
		\includegraphics{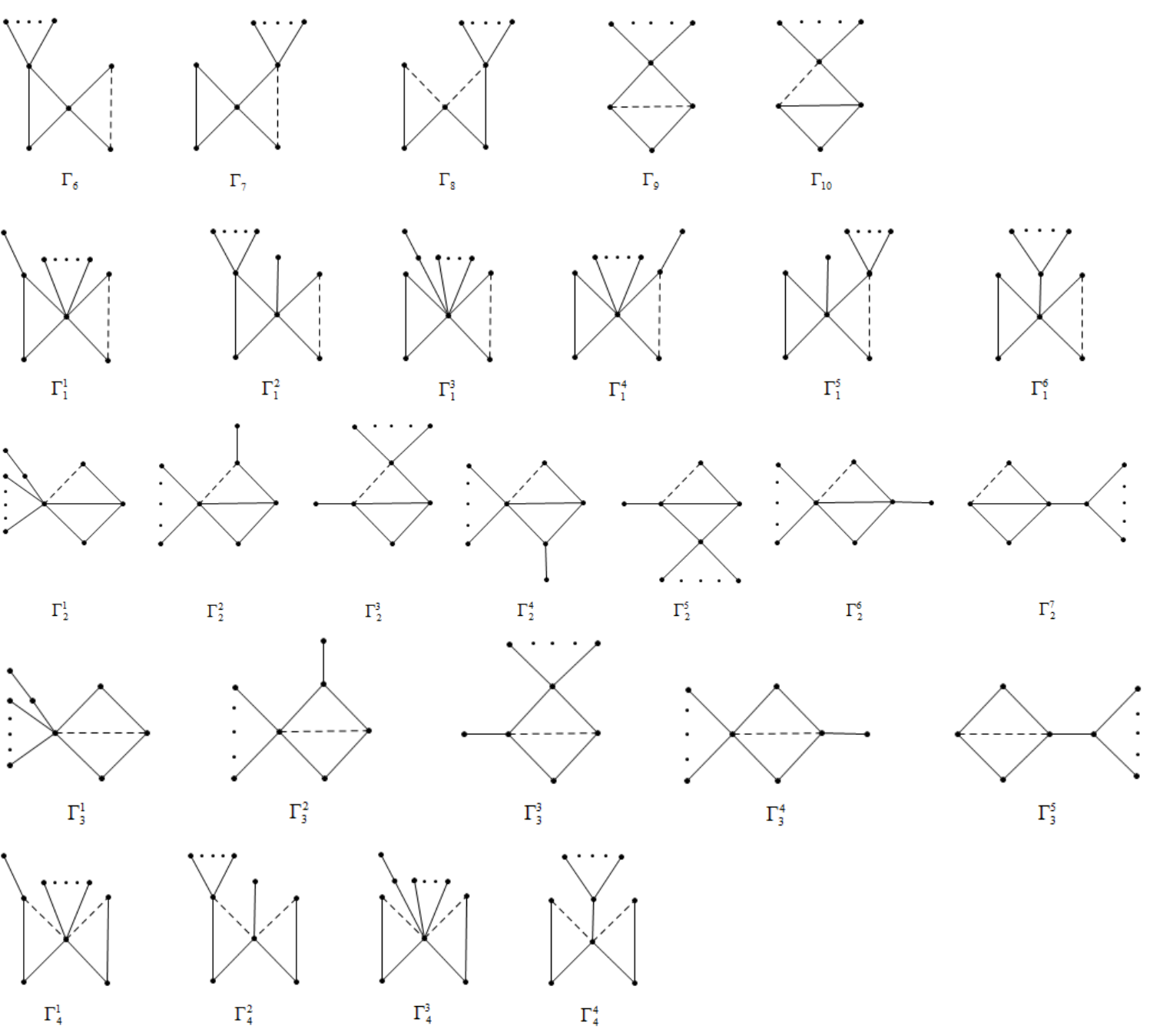}
		\makeatletter\def\@captype{figure}\makeatother 
		\vspace{-0.5cm}
		\caption{ Signed  graphs considered in the  proof of Lemma \ref{lem:inf} and Lemma \ref{lem:theta}}\label{fig:othergraphs}
	\end{center}
\end{figure}	
\begin{lemma}\label{lem:inf}
	If $\Gamma  \in \mathcal{B}_n$  is an $\infty$-type   graph and is not switching equivalent to    $\Gamma_1,$ or $ \Gamma_4$, then $\lambda(\Gamma)< \lambda(\Gamma_5)$.
	\end{lemma}
\begin{proof}
By Lemma \ref{lem:concut}, Lemma \ref{lem:gra2} and Corollary \ref{lem:concut}, it is not difficult to see that, we only need to prove that  if  $\Gamma\in \{   \Gamma_4, \ \Gamma_6,\ \Gamma_7,\ \Gamma_8,\ \Gamma_1^i,\  \Gamma_4^j\}$ , where $1\leq i\leq 6$ and $1\leq j\leq 4$ (  as shown in Figure \ref{fig:othergraphs}). By direct computation, we can prove that 
$$\lambda(\Gamma_5)>\lambda(\Gamma_6)=\max\{\lambda(\Gamma_6),\lambda(\Gamma_7), \lambda(\Gamma_8)\},$$
$$\lambda(\Gamma_5)>\lambda(\Gamma_1^1)=\max\{\lambda(\Gamma_1^1),\ldots, \lambda(\Gamma_1^6)\}, $$
and $$\lambda(\Gamma_5)>\lambda(\Gamma_4^3)>\lambda(\Gamma_4^4),\ \lambda(\Gamma_5)>\lambda(\Gamma_4^1)>\lambda(\Gamma_4^2). $$
Hence, we can get the desired result.
\end{proof}

\begin{figure}[!htb] \label{6}
	\begin{center}	 
		\includegraphics{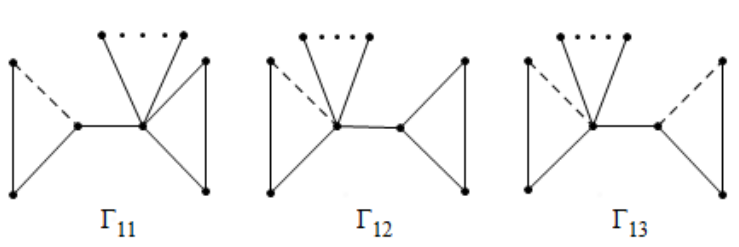}
		\makeatletter\def\@captype{figure}\makeatother 
		\vspace{-0.5cm}
		\caption{ Signed  graphs considered in proof of Lemma \ref{lem:dum} }
	\end{center}
\end{figure}	

\begin{lemma}\label{lem:dum}
	If $\Gamma $ is a dumbbell-type     unbalanced signed graph,  then $\lambda(\Gamma)< \lambda(\Gamma_5)$.
\end{lemma}

\begin{proof} Similar to the proof of Lemma \ref{lem:bibase} and Lemma \ref{lem:inf}, we know that for any $\Gamma \in \mathcal{B}_n(p,\ell,q)$, the index of $\lambda(\Gamma)\leq \max\{\lambda(\Gamma _{11}),\lambda(\Gamma _{12}),\lambda(\Gamma _{13})\}<\lambda(\Gamma _{5})$ (where $\Gamma _{11},\ \Gamma _{12},\ \Gamma _{13}$ are the signed graphs shown as in \ref{6}). 
\end{proof}

\begin{lemma}\label{lem:theta}
		If $\Gamma \in \mathcal{B}_n$  is a $\theta$-type     graph, and is not switching equivalent to  $ \Gamma_2$, $\Gamma_3$ or $\Gamma_5$, then $\lambda(\Gamma)< \lambda(\Gamma_5)$.
\end{lemma}
\begin{proof}
	It is not difficult to see that, we only need to consider the case that   $\Gamma\in \{   \Gamma_9, \ \Gamma_{10},\  \Gamma_2^i,\  \Gamma_3^j\}$ , where $1\leq i\leq 7$ and $1\leq j\leq 5$ ( as shown in Figure \ref{fig:othergraphs}). By direct computation, we can prove that 
	$$\lambda(\Gamma_5)>\lambda(\Gamma_9),\  \lambda(\Gamma_5)>\lambda(\Gamma_{10}),$$
	$$\lambda(\Gamma_5)>\lambda(\Gamma_2^1)=\max\{\lambda(\Gamma_2^1),\ldots, \lambda(\Gamma_2^4)\},$$
	$$\lambda(\Gamma_5)>\lambda(\Gamma_3^1)=\max\{\lambda(\Gamma_3^1),\lambda(\Gamma_3^3), \lambda(\Gamma_3^5)\}, \ \lambda(\Gamma_5)>\lambda(\Gamma_3^2)>\lambda(\Gamma_3^4). $$
	Hence, we can get the desired result.
\end{proof}

Combining  Lemma \ref*{lem:inf}, Lemma \ref{lem:dum} and Lemma \ref{lem:theta}, we can get the following result immediately.
\begin{theorem}  Let $\Gamma_i \in \mathcal{B}_n\ ( \mbox{ where} \  i=1,2,\ldots, 5)$ be the unbalanced signed graphs as shown in Figure  \ref{fig:T1-T5}, then

\begin{enumerate}[(1).]
	\item 	the index  $\lambda(\Gamma_i)$ is the largest root of the equation $f_i(x)=0$, where  
	\begin{align*}
	f_1(x)&=x^4-nx^2+n-5, \\
	f_2(x)&=x^4-(n+1)x^2+2n-4,  \\
	f_3(x)&=x^4-(n+1)x^2+4x+2n-8, \\
	f_4(x)&=x^3+x^2-(n-1)x-n+5, \\
	f_5(x)&=x^3-x^2-(n-2)x+n-4,
	\end{align*}
	\item  for $n\geq 36$,  we have  	$\lambda(\Gamma_1)>\lambda(\Gamma_2)>\lambda(\Gamma_3)>\lambda(\Gamma_4)>\lambda(\Gamma_5)$,
	\item if  $\displaystyle\Gamma\in \mathcal{ B}_n$   is not switching equivalent to $  \Gamma_1, \Gamma_2,\Gamma_3,\Gamma_4$ or $ \Gamma_5 $, we have  
$\lambda(\Gamma ) <\lambda (\Gamma_5 )$. 
\end{enumerate}
\end{theorem}

\newpage\section*{Appendix}	
\begin{table}[!htpb]
	\footnotesize
	\renewcommand\arraystretch{1.3}
	\centering
	\caption{The characteristic polynomials of signed graphs  in Section 4}
	\begin{tabular}{|c|l|}
		\hline
		Signed graph&Characteristic polynomial\\
		\hline
		$\Gamma_1$&$\Phi(\Gamma_1,x)=x^{n-6}(x^2-1)(x^4-nx^2+n-5)$\\
		\hline
		$\Gamma_2$&$\Phi(\Gamma_2,x)=x^{n-4}[x^4-(n+1)x^2+2n-4]$\\
		\hline
		$\Gamma_3$&$\Phi(\Gamma_3,x)=x^{n-4}[x^4-(n+1)x^2+4x+2n-8]$\\
		\hline
		$\Gamma_4$&$\Phi(\Gamma_4,x)=x^{n-6}(x+1)(x-1)^2[x^3+x^2-(n-1)x-n+5]$\\
		\hline
		$\Gamma_5$&$\Phi(\Gamma_5,x)=x^{n-5}(x+2)(x-1)[x^3-x^2-(n-2)x+n-4]$\\
		\hline
		$\Gamma_6$&$\Phi(\Gamma_6,x)=x^{n-6}(x-1)[x^5+x^4-nx^3-nx^2+(3n-15)x+n-5]$\\
		\hline
		$\Gamma_7$&$\Phi(\Gamma_7,x)=x^{n-6}(x+1)[x^5-x^4-nx^3+nx^2+(3n-15)x-n+5]$\\
		\hline
		$\Gamma_8$&$\Phi(\Gamma_8,x)=x^{n-6}(x-1)[x^5+x^4-nx^3-(n-4)x^2+(3n-11)x+n-5]$\\
		\hline
		$\Gamma_9$&$\Phi(\Gamma_9,x)=x^{n-5}(x-1)[x^4+x^3-nx^2-(n-4)x+2n-8]$\\
		\hline
		$\Gamma_{10}$&$\Phi(\Gamma_{10},x)=x^{n-5}(x-2)(x+1)[x^3+x^2-(n-2)x-n+4]$\\
		\hline
		$\Gamma_{11}$&$\Phi(\Gamma_{11},x)=x^{n-7}(x-1)^2(x+1)[x^4+x^3-(n-1)x^2-(n-1)x+2n-12]$\\
		\hline
		$\Gamma_{12}$&$\Phi(\Gamma_{12},x)=x^{n-7}(x-1)(x+1)^2[x^4-x^3-(n-1)x^2+(n-1)x+2n-12]$\\
		\hline
		$\Gamma_{13}$&$\Phi(\Gamma_{13},x)=x^{n-7}(x-1)^2[x^5+2x^4-(n-2)x^3-(2n-6)x^2+(n-3)x+2n-12]$\\
		\hline
		$\Gamma_1^1$&$\Phi(\Gamma_1^1,x)=x^{n-6}(x-1)[x^5+x^4-nx^3-nx^2+(2n-9)x+2n-11]$\\
		\hline
		$\Gamma_1^2$&$\Phi(\Gamma_1^2,x)=x^{n-6}(x-1)[x^5+x^4-nx^3-nx^2+(4n-23)x+2n-11]$\\
		\hline
		$\Gamma_1^3$&$\Phi(\Gamma_1^3,x)=x^{n-8}(x-1)^2(x+1)^2[x^4-(n-1)x^2+n-7]$\\
		\hline
		$\Gamma_1^4$&$\Phi(\Gamma_1^4,x)=x^{n-6}(x+1)[x^5-x^4-nx^3+nx^2+(2n-9)x-2n+11]$\\
		\hline
		$\Gamma_1^5$&$\Phi(\Gamma_1^5,x)=x^{n-6}(x+1)[x^5-x^4-nx^3+nx^2+(4n-23)x-2n+11]$\\
		\hline
		$\Gamma_1^6$&$\Phi(\Gamma_1^6,x)=x^{n-6}(x+1)(x-1)(x^4-nx^2+5n-29)$\\
		\hline
		$\Gamma_2^1$&$\Phi(\Gamma_2^1,x)=x^{n-6}[x^6-(n+1)x^4+(3n-7)x^2-2n+8]$\\
		\hline
		$\Gamma_2^2$&$\Phi(\Gamma_2^2,x)=x^{n-6}[x^6-(n+1)x^4+(3n-8)x^2+2x-n+5]$\\
		\hline
		$\Gamma_2^3$&$\Phi(\Gamma_2^3,x)=x^{n-6}[x^6-(n+1)x^4+(4n-14)x^2+(2n-10)x-n+5]$\\
		\hline
		$\Gamma_2^4$&$\Phi(\Gamma_2^4,x)=x^{n-6}[x^6-(n+1)x^4+(3n-8)x^2-2x-n+5]$\\
		\hline
		$\Gamma_2^5$&$\Phi(\Gamma_2^5,x)=x^{n-6}[x^6-(n+1)x^4+(4n-14)x^2-(2n-10)x-n+5]$\\
		\hline
		$\Gamma_2^6$&$\Phi(\Gamma_2^6,x)=x^{n-4}[x^4-(n+1)x^2+3n-9]$\\
		\hline
		$\Gamma_2^7$&$\Phi(\Gamma_2^7,x)=x^{n-6}[x^6-(n+1)x^4+(5n-19)x^2-4n+20]$\\
		\hline
		$\Gamma_3^1$&$\Phi(\Gamma_3^1,x)=x^{n-6}[x^6-(n+1)x^4+4x^3+(3n-11)x^2-4x-2n+12]$\\
		\hline
		$\Gamma_3^2$&$\Phi(\Gamma_3^2,x)=x^{n-6}[x^6-(n+1)x^4+4x^3+(3n-12)x^2-2x-n+6]$\\
		\hline
		$\Gamma_3^3$&$\Phi(\Gamma_3^3,x)=x^{n-6}[x^6-(n+1)x^4+4x^3+(4n-18)x^2-(2n-10)x-n+5]$\\
		\hline
		$\Gamma_3^4$&$\Phi(\Gamma_3^4,x)=x^{n-4}[x^4-(n+1)x^2+4x+3n-13]$\\
		\hline
		$\Gamma_3^5$&$\Phi(\Gamma_3^5,x)=x^{n-5}[x^5-(n+1)x^3+4x^2+(5n-23)x-4n+20]$\\
		\hline
		$\Gamma_4^1$&$\Phi(\Gamma_4^1,x)=x^{n-6}(x-1)[x^5+x^4-nx^3-(n-4)x^2+(2n-5)x+2n-11]$\\
		\hline
		$\Gamma_4^2$&$\Phi(\Gamma_4^2,x)=x^{n-6}(x-1)[x^5+x^4-nx^3-(n-4)x^2+(4n-19)x+2n-11]$\\
		\hline
		$\Gamma_4^3$&$\Phi(\Gamma_4^3,x)=x^{n-8}(x-1)^2(x+1)^2[x^4-(n-1)x^2+4x+n-7]$\\
		\hline
		$\Gamma_4^4$&$\Phi(\Gamma_4^4,x)=x^{n-7}(x-1)^2(x+1)[x^4+x^3-(n-1)x^2-(n-5)x+4n-24]$\\
		\hline
	\end{tabular}
\end{table}
		
\end{document}